\documentclass[reqno]{amsart}
\usepackage[textheight=22cm,textwidth=15cm,headsep=1cm,centering]{geometry}
\usepackage{amsmath,amssymb,amsthm,esint}
\usepackage[hidelinks]{hyperref}

\newtheorem{theorem}{Theorem}[section]
\newtheorem{lemma}[theorem]{Lemma}
\newtheorem{definition}[theorem]{Definition}
\newtheorem{proposition}[theorem]{Proposition}
\newtheorem{corollary}[theorem]{Corollary}
\newtheorem{conjecture}[theorem]{Conjecture}
\newtheorem{remark}[theorem]{Remark}

\begin{document}

\title[Inner regularity and Liouville theorems for stable solutions]
{Inner regularity and Liouville theorems for stable solutions to the mean curvature equation}
\author[Xu]{Fanheng Xu}
\address{School of Mathematics (Zhuhai), Sun Yat-Sen University, 519082, Zhuhai, P. R. China}
\email{xufh7@mail.sysu.edu.cn}
\thanks{Xu was supported by the National Natural Science Foundation of China (No.12201654)}

\begin{abstract}
Let $f\in C^1(\mathbb{R})$.
We study stable solutions $u$ of the mean curvature equation
\begin{align*}
\operatorname{div}\left( \frac{\nabla u}{\sqrt{1+|\nabla u|^2}} \right) = -f(u)
\qquad \text{in}\ \Omega \subset \mathbb{R}^n.
\end{align*}
In the local setting we prove that $\nabla u$ satisfies inner Morrey regularity $M^{p_n}$,
where
\begin{align*}
p_n := \left\{
\begin{array}{ll}
n,\qquad & \text{if}\ 2\leq n\leq 5, \\
\frac{n}{n-4\sqrt{n-1}+4},\qquad & \text{if}\ n\geq 6,
\end{array}
\right.
\end{align*}
together with the estimate
\begin{align*}
\|\nabla u\|_{M^{p_n}(B_1)} \leq C \left( 1+\|\nabla u\|_{L^1(B_2)} \right).
\end{align*}
The exponent $p_n$ is optimal for $n\leq5$, as shown by an explicit one-dimensional example.
For radial solutions we show that the symmetry center is at most a removable singularity.

Globally, we establish Liouville-type theorem:
any stable solution satisfying the growth condition
\begin{align*}
|\nabla u(x)| =
\left\{
\begin{array}{lll}
o(|x|^{-1}) \  & \text{as}\ |x|\rightarrow +\infty& \text{when}\ 2\leq n\leq 10, \\
o(|x|^{-n/2+\sqrt{n-1}+1}) \  & \text{as}\ |x|\rightarrow +\infty& \text{when}\ n\geq 11,
\end{array}
\right.
\end{align*}
must be constant.
In particular, no nonconstant radial stable solution exists in dimensions \(2\leq n\leq6\),
which highlights a global rigidity of stable radial solutions in low dimensions and extend the classical Liouville theorem of Farina and Navarro.

Several exponents appearing in our results are new for mean curvature equations, showing both similarities and differences with the corresponding theorems for semilinear equations.
\end{abstract}

\maketitle

\section{Introduction}

This article is concerned with stable solutions of the mean curvature equation
\begin{align}\label{equation}
\operatorname{div}\left( \frac{\nabla u}{\sqrt{1+|\nabla u|^2}} \right) = -f(u)
\qquad \text{in}\ \Omega \subset \mathbb{R}^n.
\end{align}
Consider the area-type functional with a potential term
\begin{align*}
\mathcal{E}[u] = \int_\Omega \left( \sqrt{1+|\nabla u|^2}+F(u) \right) dx,
\end{align*}
where $F' = f$. Computing its first variation yields
\begin{align*}
\frac{d}{dt} \mathcal{E}[u+t\phi] \bigg|_{t=0} = \int_\Omega \left( \text{div}\left( \frac{\nabla u}{\sqrt{1+|\nabla u|^2}} \right)+f(u) \right) \phi dx.
\end{align*}
Hence, a solution $u$ to equation (\ref{equation}) is precisely a critical point of $\mathcal{E}$
and (\ref{equation}) is the Euler-Lagrange equation associated with $\mathcal{E}$.

Next, we examine the second variation of $\mathcal{E}$ at $u$:
\[
\frac{d^2}{dt^2} \mathcal{E}[u+t\phi] \bigg|_{t=0}
= \int_{\Omega}
\left(
\frac{|\nabla \phi|^2}{\sqrt{1+|\nabla u|^2}} - \frac{(\nabla u \cdot \nabla \phi)^2}{(1+|\nabla u|^2)^{3/2}} - f'(u)\phi^2
\right) dx.
\]
If this quadratic form is nonnegative for every test function $\phi$, then $u$ is called a stable critical point. This motivates the following precise definition:

\begin{definition}[Stable solution]
Let $f \in C^1(\mathbb{R})$, and let $u \in W_{loc}^{1,1}(\Omega)$ be a weak solution to (\ref{equation}) in the sense that $f(u) \in L^1_{\text{loc}}(\Omega)$ and
\begin{align}\label{weak equation}
\int_{\Omega} \frac{\nabla u \cdot \nabla \varphi}{\sqrt{1+|\nabla u|^2}} dx
= \int_{\Omega} f(u) \varphi dx \qquad \forall \ \varphi \in C_c^\infty(\Omega).
\end{align}
Assume further that $f'(u) \in L^1_{\text{loc}}(\Omega)$. We say that $u$ is a stable solution of (\ref{equation}) in $\Omega$ if it satisfies the inequality
\begin{align}\label{stable}
\int_{\Omega} f'(u)\xi^2 dx
\leq
\int_{\Omega}
\left(
\frac{|\nabla \xi|^2}{\sqrt{1+|\nabla u|^2}} - \frac{(\nabla u \cdot \nabla \xi)^2}{(1+|\nabla u|^2)^{3/2}}
\right) dx \quad \forall \ \xi \in C_c^\infty(\Omega).
\end{align}
\end{definition}

\subsection{Inner Regularity}
While regularity theories are well developed for semilinear and $p$-Laplacian type equations ($p>1$),
both in the local and fractional settings
(see
\cite{Cabre Figalli Acta 2020,
Cabre Miraglio Sanchon 2022 Adv. Calc. Var.,
Dupaigne Farina 2023 CCM,
Erneta 2023 CPAA,
Peng Zhang Zhou AnalPDE 2024,
Erneta 2024 JDE, Erneta 2025 J. Reine Angew. Math.,
Sanz-Perela 2025 AMPA, 
Zhang 2025 DCDS}
and the references therein),
the corresponding theory for mean curvature operator remains largely unexplored (except for $f=0$).


Our main result concerns the integrability regularity of the gradient, described in terms of Morrey spaces. We use the following definition:

\begin{definition}[Morrey Spaces]
Let $1\leq p \leq \infty$.
A function $g \in L_{loc}^1(\Omega)$ is said to belong to the
Morrey space $M^p(\Omega)$ if there exists a constant $K>0$ such that
\begin{align}\label{d5a0e755b}
\int_{B_\rho(x_0)\cap\Omega} |g(x)|dx
\leq K\rho^{n(1-1/p)}
\end{align}
for all $x_0\in\Omega$ and all radii
$0<\rho\leq \operatorname{diam}(\Omega)$.
We define the norm $\| g \|_{M^p(\Omega)}$ to be the infimum of the constants $K$ satisfying (\ref{d5a0e755b}).
\end{definition}

Since the result is local, we state it in the ball.

\begin{theorem}\label{theorem regularity}
Let $f\in C^1(\mathbb{R})$
and $u\in W_{loc}^{3,1}(B_2)$ be a stable solution of (\ref{equation}) in $B_2$.
Define the exponent
\begin{align*}
p_n := \left\{
\begin{array}{ll}
n\qquad & \text{if}\ 2\leq n\leq 5, \\
\frac{n}{n-4\sqrt{n-1}+4}\qquad & \text{if}\ n\geq 6.
\end{array}
\right.
\end{align*}
Then
\begin{align}\label{gradL1 ctrl GradMn}
\|\nabla u\|_{M^{p_n}(B_1)}
\leq C \left( 1+\|\nabla u\|_{L^1(B_2)} \right),
\end{align}
where $C$ is a dimensional constant.
\end{theorem}

A useful consequence of Theorem~\ref{theorem regularity} is the following compactness result:
Let ${v_\varepsilon}$ be a sequence of stable solutions of (\ref{equation}) in $B_2$ converging to $u$ in $W^{1,1}(B_2)$.
While such convergence only yields a priori $\nabla u\in L^1$, estimate (\ref{gradL1 ctrl GradMn}) allows us to upgrade the regularity of the limit, showing that $\nabla u\in M^{p_n}(B_1)$.
To the best of our knowledge, this is the first result providing regularity improvement (from $W^{1,1}$ to Morrey spaces) for stable solutions to the mean curvature equation.

For $2 \leq n \leq 5$, Theorem~\ref{theorem regularity} yields $\nabla u \in M^n$.
The exponent $n$ is critical: any improvement to $\nabla u \in M^{n+\alpha}$, $\alpha > 0$, would imply H\"older continuity of $u$ by the Morrey embedding theorem (see Theorem~7.19 in \cite{Gilbarg Trudinger Berlin 1998}).

In particular,
for semilinear equation
\begin{align}\label{semilinear equation}
\Delta u = -f(u),
\end{align}
Cabr\'e, Figalli, Ros-Oton and Serra \cite{Cabre Figalli Acta 2020} established the following regularity result:
\begin{theorem}[Cabr\'e et al. \cite{Cabre Figalli Acta 2020}]
Let $f \in C^{0,1}(\mathbb{R})$ be nonnegative, and $u\in C^2(B_2)$ be a stable solution of (\ref{semilinear equation}) in $B_2$. If $n\leq 9$, then
\begin{align*}
\| u \|_{C^\alpha(\overline{B}_1)} \leq C \| u \|_{L^1(B_2)},
\end{align*}
where $\alpha \in (0,1)$ and $C$ is a dimensional constant.
\end{theorem}

This raises the natural question whether a similar H\"older improvement holds for mean curvature equation. 
The following result shows that such an improvement cannot, in general, be expected solely from the stability condition even in low dimensions.

\begin{theorem}\label{theorem bad regularity solution}
Let $n\geq 2$ and $\alpha>0$.
There exists a one-dimensional stable solution $u$ of equation (\ref{equation}) in $B_1$
(for a suitable nonlinearity $f$)
such that
\begin{align*}
\nabla u \notin M^{n+\alpha}(B_1).
\end{align*}
\end{theorem}

This demonstrates that the exponent $p_n$ in Theorem~\ref{theorem regularity} is at least optimal for dimensions $2\leq n\leq 5$, in the sense that no stronger Morrey regularity of $\nabla u$ can be derived from the stability condition alone.

We now turn to the case $n\geq 6$.
In this regime, the Morrey exponent
$$
p_n=\frac{n}{n-4\sqrt{n-1}+4},
$$
provided by Theorem~\ref{theorem regularity} corresponds to a very singular Morrey class.
Establishing the optimality of this exponent would require the construction of stable solutions exhibiting very poor regularity.
However, such examples appear to be difficult to obtain.
Note that the one-dimensional stable solutions used in Theorem~\ref{theorem bad regularity solution} cannot serve this purpose in high dimensions, since strongly singular one-dimensional profiles typically fail to satisfy the local integrability required in right-hand side (RHS) of (\ref{gradL1 ctrl GradMn}).

A natural alternative is to consider radially symmetric solutions with a singularity at the symmetry center.
For semilinear equation (\ref{semilinear equation}), it is well known that there exists stable radial solutions with isolated singularities when $n\geq 10$; see, for instance, \cite{Cabre Figalli Acta 2020} and references therein.
In contrast, we show that such behavior cannot occur for equation (\ref{equation}).

\begin{proposition}\label{proposition 1}
Let $n\geq 2$, $f\in C^1(\mathbb{R})$, and let $u\in W_{loc}^{3,1}(B_2)$ be a stable solution of (\ref{equation}).
Assume that $u$ is radially symmetric with respect to the origin $0$.
Then, for every $\alpha \in (0,1)$, there exists $K=K(n,\alpha)$ such that
\begin{align}\label{ef434ad8b}
\int_{B_\rho(0)} |\nabla u|
\leq K \rho^{n-1+\alpha}
\qquad \forall \ 0<\rho<1.
\end{align}
\end{proposition}

Although estimate (\ref{ef434ad8b}) is a purely local condition at the single point $0$
and does not directly yield continuity through standard Morrey embeddings, the radial symmetry provides additional structure that allows one to control the behavior at the origin.

\begin{corollary}\label{corollary 1}
Under the assumptions of Proposition~\ref{proposition 1}, the origin is at most a removable singularity of $u$.
\end{corollary}

In particular, any radially symmetric stable solution of \eqref{equation} admits a continuous representative at the origin.
This contrasts sharply with the semilinear case, where isolated singular stable solutions do exist.
Determining the sharpness of $p_n$ for $n\geq 6$ remains an interesting open problem, potentially requiring new constructions of singular stable solutions.

\subsection{Liouville-type Theorem}

For the global case $\Omega = \mathbb{R}^n$, we may give the following Liouville-type theorem:
\begin{theorem}\label{theorem liouville}
Let $f\in C^1(\mathbb{R})$
and $u$ be a $W_{loc}^{3,1}(\mathbb{R}^n)$ stable solution of equation (\ref{equation}).
Define
\begin{align*}
q_n :=
\left\{
\begin{array}{ll}
-1 \quad & \text{if}\ 2\leq n\leq 10, \\
-n/2+\sqrt{n-1}+1 \quad & \text{if}\ n\geq 11.
\end{array}
\right.
\end{align*}

\begin{enumerate}
\item
If $u$ is nonconstant, then
\begin{align}\label{liouville theorem nonconstant bound}
\fint_{B_{2R}\backslash B_{R}} |\nabla u|^2 dx \geq c R^{2q_n} \qquad \forall \ R\geq R_0
\end{align}
for some $c>0$ and $R_0>2$.
\item
In particular, if $u$ satisfies the growth condition
\begin{align}\label{liouville theorem rigidity condition}
|\nabla u(x)| = o(|x|^{q_n}), \qquad \text{as}\ |x|\rightarrow+\infty.
\end{align}
then $u$ must be a constant.
\end{enumerate}
\end{theorem}

Our approach can also be applied to radial stable solutions to obtain a stronger version of Theorem~\ref{theorem liouville}.
\begin{theorem}\label{liouville theorem radial}
Let $f\in C^1(\mathbb{R})$ and $u$ be a $W_{loc}^{3,1}(\mathbb{R}^n)$ stable radial solution of equation (\ref{equation}).
\begin{enumerate}
\item
If $u$ is nonconstant, then
\begin{align}\label{liouville theorem radial nonconstant bound}
\fint_{B_{2R}\backslash B_{R}} \frac{|\nabla u|^2}{(1+|\nabla u|^2)^{3/2}} dx \geq c R^{-n+2\sqrt{n-1}+2} \qquad \forall \ R\geq R_0
\end{align}
for some $c>0$ and $R_0>2$.
\item
In particular, if one of the following condition is satisfied
\begin{itemize}
\item $2\leq n\leq 6$,
\item $n \geq 7$ and
\begin{align*}
|\nabla u(x)| = o(|x|^{-n/2+\sqrt{n-1}+1})\qquad \text{as}\ |x|\rightarrow+\infty,
\end{align*}
\item $n \geq 7$ and
\begin{align*}
|\nabla u(x)|^{-1} = o(|x|^{-n+2\sqrt{n-1}+2})\qquad \text{as}\ |x|\rightarrow+\infty,
\end{align*}
\end{itemize}
then $u$ must be a constant.
\end{enumerate}
\end{theorem}

In this field, Farina and Navarro studied the stable radial solution of equation (\ref{equation}).
They proved that:
\begin{theorem}[Theorem 1.2 in \cite{Farina Navarro DCDS 2020}]\label{liouville theorem Farina}
Let $u$ be a $C^{2}(\mathbb{R}^n)$ stable radial solution of equation (\ref{equation}).
Define
\begin{align*}
\overline{q}_n=
\left\{
\begin{array}{ll}
-n+2\sqrt{n-1}+3 \quad & \text{if}\ 2\leq n\leq 6, \\
-n/2+\sqrt{n-1}+2 \quad & \text{if}\ n\geq 7.
\end{array}
\right.
\end{align*}

\begin{enumerate}
\item
If $u$ is nonconstant, then
\begin{align*}
|u(R)| \geq c R^{\overline{q}_n}
\end{align*}
for some $c>0$ and $R_0>2$.
\item
In particular, if $u$ satisfies the growth condition
\begin{align*}
|u(x)| = o(|x|^{\overline{q}_n}), \qquad \text{as}\ |x|\rightarrow+\infty.
\end{align*}
then $u$ must be a constant.
\end{enumerate}
\end{theorem}

It is instructive to compare Theorems~\ref{theorem liouville}, \ref{liouville theorem radial} with the classical Liouville Theorem~\ref{liouville theorem Farina} of Farina and Navarro
\footnote{We note that their result distinguishes the case $n=10$ (corresponding to $\overline{q}_{10}=0$) where a logarithmic lower bound appears. For the convenience of discussion we consider logarithm as a 0-power.}.
Both results provide polynomial bounds for stable solutions.
Since our results concern $\nabla u$ while Theorem~\ref{liouville theorem Farina} provides estimates for $u$, we compare our exponents with $\overline{q}_n - 1$ to account for the order of differentiation.

We first compare
the radial case.
Theorem~\ref{liouville theorem radial} yields a stronger result.
In dimensions $n\geq 7$, the lower bound in (\ref{liouville theorem radial nonconstant bound}) recovers precisely the critical exponent $\overline{q}_n - 1$ appearing in Theorem~\ref{liouville theorem Farina}.
Moreover, in contrast to previous results, Theorem~\ref{liouville theorem radial} also provides an upper growth condition, leading to a two-sided Liouville-type characterization for radial stable solutions.

Most notably, in low dimensions $2\leq n\leq 6$, our result shows that no nonconstant radial stable solutions exist at all, without imposing any growth assumption at infinity.
This phenomenon does not follow from either Theorem~\ref{theorem liouville} or Theorem~\ref{liouville theorem Farina}, and highlights a global rigidity of stable radial solutions in low dimensions.

We now turn to compare
Theorem~\ref{theorem liouville} with Theorem~\ref{liouville theorem Farina}.
As we explain before, Theorem~\ref{liouville theorem radial} shows that, for $2\leq n\leq 6$, every stable radial solution must be constant.
Consequently, in this dimensional range, any Liouville-type result relying on polynomial growth assumptions, such as Theorem~\ref{liouville theorem Farina}, cannot be optimal.
For this reason, the comparison of growth exponents is only meaningful in dimensions $n\geq 7$.
Up to a loss of one order in the exponent,
a direct computation shows that
\begin{align*}
q_n
\left\{
\begin{array}{l}
< \overline{q}_n - 1 \qquad \text{when}\ 7\leq n\leq 9,\\
= \overline{q}_n - 1 \qquad \text{when}\ n\geq 10.
\end{array}
\right.
\end{align*}
Hence the growth bound in (\ref{liouville theorem nonconstant bound}) is slightly weaker than that of Farina and Navarro when $7\leq n\leq 9$.
However, Theorem~\ref{theorem liouville} applies to general (possibly nonradial) stable solutions,
while Theorem~\ref{liouville theorem Farina} is restricted to the radial setting.

It is also natural to compare our Liouville-type results with the corresponding ones for the semilinear equation (\ref{semilinear equation}),
whose global stable solutions have been extensively studied.

When $n \geq 11$, an interesting coincidence occurs:
the growth lower bounds obtained in our results for general stable solutions in Theorem~\ref{theorem liouville} and for radial stable solutions in Theorem~\ref{liouville theorem radial} coincide with those known for the semilinear equation (\ref{semilinear equation}),
both in the radial setting \cite{Villegas JMPA 2007} and nonradial setting \cite{Peng Zhang Zhou AnalPDE 2024},
after adjusting for the order of differentiation.
In this sense,
in high dimensions the critical growth rates prescribed by the two equations are identical.

In lower dimensions, the situation becomes more delicate.
For general stable solutions of the semilinear equation (\ref{semilinear equation}) with $f\geq 0$,
Dupaigne and Farina \cite{Dupaigne Farina 2022 AnalPDE} showed that if $n \leq 9$ and $u(x) \geq -C\log(2+|x|)$,
or if $ n=10 $ and $u \geq -C$ for some constant $ C>0 $,
then $u$ must be constant.
If one regards logarithmic growth as a zero-order behavior, this threshold is consistent with the exponent $q_n+1$ appearing in our Liouville theorem for general solutions. In this sense, the growth bounds for general stable solutions in low dimensions still coincide for the two equations.

A genuine difference arises when restricting to radial solutions in low dimensions. Our results show that, for $2 \leq n \leq 6$, equation (\ref{equation}) admits no nontrivial global radial stable solutions.
In contrast, the examples constructed by Villegas \cite{Villegas JMPA 2007} show that, for the semilinear equation (\ref{semilinear equation}), nontrivial radial stable solutions exist for $2 \leq n \leq 9$.
This discrepancy may suggest the existence of nonradial global stable solutions of (\ref{equation}) in low dimensions exhibiting behavior distinct from the radial ones. At present, however, we do not have the analytical tools to construct such solutions. For examples of stable radial solutions of the equation (\ref{equation}), we refer to \cite{Farina Navarro DCDS 2020}.

\subsection{Some History}

The appearance of a critical dimension in our results is not accidental.
Similar dimension-dependent phenomena have played a central role in the study of the mean curvature operator and minimal surfaces.

A classical example is the minimal surface equation
\begin{align}\label{minimal surface equation}
\operatorname{div}\left( \frac{\nabla u}{\sqrt{1+|\nabla u|^2}} \right) = 0
\quad \text{in}\ \mathbb{R}^n,
\end{align}
which corresponds to minimal graphs in \(\mathbb{R}^{n+1}\).

Bernstein \cite{Bernstein MathZ 1927} showed that in dimension \(n=2\),
every entire solution of (\ref{minimal surface equation}) must be affine linear function.
Through a sequence of fundamental works by Fleming \cite{Fleming RCMP 1962},
De Giorgi \cite{DeGiorgi ASNSP 1965},
Almgren \cite{Almgren 1966 Annals},
and Simons \cite{Simons Annals 1968},
this rigidity result was extended to all dimensions \(n \leq 7\).
In sharp contrast, for \(n \geq 8\), Bombieri, De Giorgi, and Giusti \cite{Bombieri Giorgi Giusti Invent 1969}
constructed non-affine entire minimal graphs, showing that the Bernstein property fails in high dimensions.

A classical insight, originating in the work of Fleming \cite{Fleming RCMP 1962} and De Giorgi \cite{DeGiorgi ASNSP 1965}, links entire minimal graphs to minimal cones:
the falsity of the extension of the Bernstein theorem in $\mathbb{R}^n$ would imply, via a blow-down procedure, the existence of a minimal cone in $\mathbb{R}^{n-1}$.
This observation reduces the Bernstein problem for entire graphs
to the question of existence of singular minimal cones.
It is known that there exist no nontrivial codimension-one stable minimal cones in $\mathbb{R}^{n}$ for $n\leq 7$. The fundamental example appears in dimension $n=8$: the Simons cone
$$\{x\in\mathbb{R}^{8} : x_1^2+x_2^2+x_3^2+x_4^2 = x_5^2+x_6^2+x_7^2+x_8^2 \},$$
which is stable (in fact, area-minimizing; see \cite{Simons Annals 1968, Bombieri Giorgi Giusti Invent 1969}).

This geometric picture has a natural PDE counterpart in the regularity theory
of stable solutions.
Just as stable minimal hypersurfaces in $\mathbb{R}^{n+1}$ may develop singularities when $n\geq 8$,
the question whether they remain smooth for all $n\leq 7$ has been a central open problem in geometric analysis.
The case $n=3$ was settled affirmatively by Fischer-Colbrie and Schoen \cite{Fischer-Colbrie Schoen CPAM 1980} and by do Carmo and Peng \cite{do-Carmo Peng BAMS 1979}.
More recently, substantial progress has been made in higher dimensions.
Chodosh and Li \cite{Chodosh Li Acta 2024} proved that any complete, two-sided, stable minimal hypersurface in $\mathbb{R}^{4}$ must be a hyperplane;
Very recently, Bellettini \cite{Bellettini Invent 2025} obtained an alternative proof of the Schoen-Simon-Yau curvature estimates and extended the associated Bernstein-type theorems to stable minimal hypersurfaces in $\mathbb{R}^{n+1}$ for $n\leq 6$.

Consider the Allen-Cahn equation
\begin{align}\label{Allen-Cahn equation}
\Delta u = u^{3} - u \qquad \text{in}\ \mathbb{R}^n.
\end{align}

By the $\Gamma$-convergence results of Modica and Mortola \cite{Modica Mortola BUMIB 1977} and by Kohn and Sternberg \cite{Kohn Sternberg PRSES 1989},
the level sets of solutions converge (under suitable scaling) to minimal hypersurfaces.
Combined with the dimension-dependent regularity theory for stable minimal hypersurfaces,
this variational picture led De Giorgi to formulate the following conjecture in 1979.

\begin{conjecture}[De Giorgi]
Let $u$ be a bounded solution of (\ref{Allen-Cahn equation}) in $\mathbb{R}^n$ that is monotone in one direction ($u_{x_n}>0$). If $n\leq 8$, then $u$ is one-dimensional.
\end{conjecture}

This conjecture has been fully resolved for $n \leq 8$:
it holds for $n=2$ (see \cite{Ghoussoub Gui MathAnn 1998}),
$n=3$ (see \cite{Ambrosio Cabre JAMS 2000} ),
and $n=4,\dots,8$ (see \cite{Savin 2009 Annals}).
Counterexamples exist for $n \geq 9$ (see \cite{del Pino Kowalczyk Wei 2011 Annals}).

In recent years, increasing attention has been devoted to symmetry questions for stable solutions,
often referred as \emph{stable De Giorgi problem} (see the survey by Chan and Wei \cite{Chan Wei SciChianMath 2018}),
asks whether bounded stable solutions to semilinear equations such as $\mathop{L}u = f(u)$ in $\mathbb{R}^n$ must be one-dimensional in low dimensions.
Here the stability condition is motivated by the fact that monotone solutions of such equations are always stable (see \cite[Corollary 4.3]{Alberti Ambrosio Cabre 2001 ActaAppl.Math}).
Significant progress has been made on this stable version:

\begin{itemize}
\item
For the classical Laplacian and $f=u-u^3$,
the one-dimensional symmetry of stable solutions holds in \(n=2\)
(see \cite{Berestycki Caffarelli Nirenberg ASNSP 1997, Ghoussoub Gui MathAnn 1998, Ambrosio Cabre JAMS 2000}),
while counterexamples exist in dimension $n=8$
(see \cite{Pacard Wei JFA 2013, Liu Wang Wei JMPA 2017}).

\item
For fractional Laplacian $(-\Delta)^{s}$,
the symmetry result was proved for
$n=2$,
$s=1/2$ and $f=u-u^3$ in
\cite{Cabre Sola CPAM 2005},
and later for $s\in(0,1)$ in
\cite{Sire Valdinoci JFA 2009, Dipierro Serra Valdinoci AJM 2020}.
Figalli and Serra \cite{Figalli Serra Invent 2020} proved the case $n=3$, $s=1/2$ and $f\in C^{0,\alpha}$.

\item
For a large class of quasilinear elliptic operators (including the mean curvature operator)
and locally Lipschitz function $f$,
the symmetry result was proved in $n=2$ by
Farina, Sciunzi and Valdinoci
\cite{Farina Sciunzi Valdinoci ASNSP 2008}.

\end{itemize}

These results suggest that one-dimensional symmetry may be more natural than radial symmetry for stable solutions in low dimensions.
This observation motivates our approach in low dimensions:
\begin{itemize}
\item
we focus on the construction of stable solutions exhibiting one-dimensional symmetry,
rather than radial symmetry in Theorem~\ref{theorem bad regularity solution};

\item
we establish nonexistence result for nonconstant radial stable solutions in Theorem~\ref{liouville theorem radial}.
\end{itemize}

The paper is organized as follows. In Section 2 we establish a key technical lemma that forms the foundation of the subsequent analysis. Section 3 is devoted to the proof of the main regularity theorem. In Section 4 we prove the corresponding Liouville-type theorem. Finally, the Appendix contains the proof of a technical auxiliary lemma used in the main arguments.

\paragraph{\bf Notation.}
We write $\nabla u$ for the gradient and $D^2 u = (u_{ij})$ for the Hessian matrix.
We set $z := |\nabla u|^2$. The symbol $\delta_{ij}$ stands for the Kronecker delta.
Subscripts on functions denote partial derivatives with respect to the spatial variables, while subscripts on $x = (x_1,\dots,x_n)$ indicate its components.
The letter $C$ denotes a positive constant that may vary from line to line.
$B_\rho(x) $ denotes the open ball of radius $\rho$ centered at $ x $, and $B_\rho := B_\rho(0)$.

\section{A Key Gradient Estimation}

The purpose of this section is to establish a gradient estimate (Lemma~\ref{lemma key grad est} below) that serves as the common foundation for both our interior regularity theorem~\ref{theorem regularity} and the Liouville-type theorem~\ref{theorem liouville}.

The estimate (\ref{key grad est}) is scale-adapted in a strong sense that once a suitable lower bound on the remainder term $L_1+L_2+L_3$ is available under the dimensional condition, the inequality yields both results by different limiting procedures:

\begin{itemize}
\item
Local regularity (Morrey-type): fix \(R > 0\) and let \(\rho \to 0\) to obtain a scale estimate for $\nabla u$;
\item
Global rigidity (Liouville-type): fix \(\rho > 0\) and let \(R \to \infty\), forcing \(\nabla u \equiv 0\) on \(\mathbb{R}^n\) under appropriate conditions.
\end{itemize}

We now state the precise result. For notational convenience, we set $z = |\nabla u|^2$.

\begin{lemma}\label{lemma key grad est}
Let $a>0$,
$0<\rho<R$.
Suppose $u$ is a $W^{3,1}_{loc}(B_{2R})$ stable solution of (\ref{equation}).
Let \(\phi\in C_c^\infty(B_{2R})\) satisfy \(\phi\equiv 1\) in \(B_\rho\).
Then $u$ satisfies
\begin{align}\label{key grad est}
&\frac{C}{\rho^{2a}} \int_{B_\rho} \frac{z^2}{(1+z)^{3/2}} \phi^2 dx
+ \int_{B_{2R}\backslash B_\rho} (L_1+L_2+L_3) \phi^2 dx \notag\\
&\leq
\int_{B_{2R}\backslash B_\rho}
\bigg(
-\frac{(x \cdot \nabla u )^2(\nabla u \cdot \nabla \phi)^2}{|x|^{2a}(1+z)^{3/2}}
+ \frac{(x \cdot \nabla u )^2 |\nabla \phi|^2}{|x|^{2a}\sqrt{1+z}} \notag\\
&\qquad\qquad\qquad+\frac{2a (x \cdot \nabla u )^3(\nabla u \cdot \nabla \phi)\phi}{|x|^{2(a+1)}(1+z)^{3/2}}
-\frac{2a (x \cdot \nabla u )^2(x \cdot \nabla \phi)\phi}{|x|^{2(1+a)}\sqrt{1+z}} \notag\\
&\qquad\qquad\qquad+\frac{2(2+z)(x \cdot \nabla u )(\nabla u \cdot \nabla \phi)\phi}{|x|^{2a}(1+z)^{3/2}}
- \frac{z(x\cdot\nabla \phi)}{|x|^{2a}\sqrt{1+z}},
\bigg)dx.
\end{align}
where $C$ is dimensional constant, and 
\begin{align*}
&L_1:=\frac{a^2 (x \cdot \nabla u )^4}{|x|^{2a+4}(1+z)^{3/2}},\\
&L_2:=\frac{a\left((2-a) z+4-a \right)(x \cdot \nabla u )^2}{|x|^{2a+2}(1+z)^{3/2}},\\
&L_3:=\frac{z\left((n-1-2a)z+(n-2-2a) \right)}{|x|^{2a}(1+z)^{3/2}}.
\end{align*}
\end{lemma}




\begin{remark}
The assumption $u\in W^{3,1}_{\mathrm{loc}}(B_{2R})$ is purely technical: it is imposed only to guarantee that the test function $\varphi$ belongs to $W^{1,1}_0(B_{2R})$. All higher-order derivatives of $u$ are eliminated by integration by parts, so the final estimates  depend only on the first derivatives of $u$ and on the structure of the equation. One may expect the assumption $W^{3,1}$ to be relaxed by density or approximation arguments.
\end{remark}

\begin{proof}[proof of Lemma~\ref{lemma key grad est}]

The weak formulation $(\ref{weak equation})$ admits test functions in $W^{1,1}_0(2R)$ by a standard density argument.
Indeed, since $u\in W^{3,1}_{loc}(B_{2R})$ and $\eta\in C_c^\infty(B_{2R})$, the function
\begin{align*}
\varphi=\operatorname{div} \bigl(x(x\cdot\nabla u)\eta^2\bigr)
\end{align*}
belongs to $W^{1,1}_0(B_{2R})$.
Moreover, the vector field $\nabla u/\sqrt{1+|\nabla u|^2}$ is locally bounded,
then by a standard density argument (since $C_c^\infty$ is dense in $W^{1,1}_0$), the weak formulation (\ref{weak equation}) holds for such $\varphi$.

Recall that $f \in C^1$. An integration by parts yields
\begin{align*}
\int_{B_{2R}} f(u) \varphi dx
= - \int_{B_{2R}} f'(u) (x\cdot \nabla u )^2 \eta^2 dx
\end{align*}
In the following, indices always range from $1$ to $n$,
and we adopt the Einstein summation convention: repeated indices are implicitly summed over. In particular,
$$
x_i u_i = x \cdot \nabla u,
\qquad
u_j u_{ij} = \frac{z_i}{2}
\qquad
x_jx_k u_{ij}u_{ik}= |D^2 u\, x|^2.
$$
Using integration by parts again, we obtain
\begin{align}\label{xtest lhs}
&\int_{B_{2R}} \frac{\nabla u \cdot \nabla \varphi}{\sqrt{1+z}} dx \notag\\
=& -\int_{B_{2R}} \sum_{i,j=1}^n \partial_i \left( \frac{u_j}{\sqrt{1+z}} \right) \partial_j \left( x_i(x\cdot \nabla u ) \eta^2 \right) dx \notag\\
= & -\int_{B_{2R}} \left( \frac{u_{ij}}{\sqrt{1+z}} - \frac{1}{2}\frac{u_j z_i}{(1+z)^{3/2}} \right) \notag\\
&\qquad\qquad\qquad \times \left( \delta_{ij}(x\cdot\nabla u ) \eta^2
+ x_iu_j \eta^2
+ x_ix_k u_{kj} \eta^2
+ 2x_i(x\cdot\nabla u ) \eta\eta_j
\right) dx\notag\\
=&-\int_{B_{2R}} \bigg( \frac{\Delta u (x\cdot\nabla u )\eta^2}{\sqrt{1+z}}
+ \frac{(x\cdot\nabla z)\eta^2}{2\sqrt{1+z}}
+ \frac{|D^2 u\, x|^2 \eta^2}{\sqrt{1+z}}
+ \frac{2((D^2 u\, x)\cdot \nabla \eta)(x\cdot\nabla u )\eta}{\sqrt{1+z}} \notag\\
&\qquad\qquad-\frac{(x\cdot\nabla u )(\nabla u \cdot \nabla z)\eta^2}{2(1+z)^{3/2}}
-\frac{(x \cdot \nabla z)z\eta^2}{2(1+z)^{3/2}}
-\frac{(x \cdot \nabla z)^2\eta^2}{4(1+z)^{3/2}} \notag\\
&\qquad\qquad-\frac{(x \cdot \nabla z)(\nabla u \cdot \nabla \eta)(x \cdot \nabla u )\eta}{(1+z)^{3/2}}
\bigg)dx.
\end{align}
Then using weak formulation (\ref{weak equation}), we obtain
\begin{align}\label{weak equation est test}
\int_{B_{2R}} f'(u) (x\cdot \nabla u )^2 \varphi dx = -\text{RHS of}\ (\ref{xtest lhs})
\end{align}

Define
\begin{align*}
\xi = (x \cdot \nabla u) \eta,
\end{align*}
which is a valid test function in (\ref{stable}) by a density argument.
We have
\begin{align*}
\nabla \xi = \eta \nabla u+\eta (D^2 u\, x)+(x \cdot \nabla u) \nabla \eta.
\end{align*}
It follows that
\begin{align}\label{8a89d914a}
|\nabla \xi|^2 =& z \eta^2+|D^2 u\, x|^2\eta^2+(x\cdot\nabla u )^2 |\nabla \eta|^2 \notag\\
&+(x\cdot \nabla z) \eta^2
+ 2(x\cdot\nabla u )(\nabla u \cdot \nabla \eta)\eta
+ 2((D^2 u\, x) \cdot \nabla \eta)(x\cdot\nabla u )\eta.
\end{align}
and
\begin{align}\label{c5445e7c3}
(\nabla u \cdot \nabla \xi)^2 = &
z^2 \eta^2
+ \frac{1}{4}(x\cdot\nabla z)^2\eta^2
+ (x\cdot\nabla u )^2(\nabla u \cdot \nabla \eta)^2
+ (x\cdot\nabla z) z \eta^2 \notag\\
&+2(x\cdot\nabla u )(\nabla u \cdot\nabla \eta) z \eta
+ (x\cdot\nabla z)(x\cdot\nabla u )(\nabla u \cdot \nabla \eta)\eta
\end{align}

Substituting (\ref{8a89d914a}) and (\ref{c5445e7c3}) into (\ref{stable}), we obtain
\begin{align}\label{stable est test}
&\int_{B_{2R}} f'(u)(x\cdot\nabla u )^2 \eta^2 dx \notag\\
&\leq \int_{B_{2R}} \bigg(
\frac{|D^2 u\, x|^2 \eta^2}{\sqrt{1+z}}
+\frac{2 (x\cdot\nabla u ) (\nabla u \cdot \nabla \eta)\eta}{(1+z)^{3/2}}
-\frac{(x\cdot\nabla u ) (\nabla u \cdot \nabla \eta) (x\cdot \nabla z)\eta}{(1+z)^{3/2}} \notag\\
&\qquad\qquad
+\frac{2 (x\cdot\nabla u ) ((D^2 u\, x) \cdot \nabla\eta)\eta}{\sqrt{1+z}}
-\frac{(x\cdot\nabla u )^2 (\nabla u \cdot \nabla \eta)^2}{(1+z)^{3/2}} \notag\\
&\qquad\qquad+\frac{(x\cdot\nabla u )^2 |\nabla \eta|^2}{\sqrt{1+z}}
+\frac{(x\cdot \nabla z)\eta^2}{(1+z)^{3/2}}
-\frac{(x\cdot \nabla z)^2\eta^2}{4 (1+z)^{3/2}}
+\frac{z\eta^2}{(1+z)^{3/2}} \bigg) dx.
\end{align}
Combining (\ref{weak equation est test}) and (\ref{stable est test}), we have
\begin{align}\label{def1a7069}
&\int_{B_{2R}} \bigg(
\frac{\Delta u (x\cdot \nabla u )\eta^2}{\sqrt{1+z}}
-\frac{(x\cdot \nabla u ) (\nabla u \cdot \nabla z)\eta^2}{2 (1+z)^{3/2}}
-\frac{2 (x\cdot \nabla u ) (\nabla u \cdot \nabla \eta)\eta}{(1+z)^{3/2}} \notag\\
&\qquad+\frac{(x\cdot \nabla u )^2 (\nabla u \cdot \nabla \eta)^2}{(1+z)^{3/2}}
-\frac{(x\cdot \nabla u )^2|\nabla \eta|^2}{\sqrt{1+z}}
-\frac{(x \cdot \nabla z)\eta^2}{2 (1+z)^{3/2}}
-\frac{z\eta^2}{(1+z)^{3/2}}
\bigg) dx
\leq 0.
\end{align}

Next, we eliminate all second-order derivatives of $u$ in (\ref{def1a7069}).
Expanding the divergence operator directly gives
\begin{align*}
\operatorname{div}\left( \frac{\nabla u}{\sqrt{1+z}} (x \cdot \nabla u ) \right)
=\operatorname{div}\left( \frac{\nabla u}{\sqrt{1+z}} \right)(x \cdot \nabla u )
+ \frac{z}{\sqrt{1+z}}
+ \frac{x\cdot\nabla z}{2\sqrt{1+z}}
\end{align*}
and
\begin{align*}
\operatorname{div}\left( x \sqrt{1+z} \right) = n\sqrt{1+z}+\frac{x \cdot \nabla z}{2\sqrt{1+z}}
\end{align*}
Cancelling the identical last term on the right-hand sides of the two equalities above, we obtain
\begin{align*}
\operatorname{div}\left( \frac{\nabla u}{\sqrt{1+z}} \right)(x \cdot \nabla u )
= \operatorname{div}\left( \frac{\nabla u}{\sqrt{1+z}} (x \cdot \nabla u ) - x \sqrt{1+z} \right)
- \frac{z}{\sqrt{1+z}}+n\sqrt{1+z}
\end{align*}
Using the above equality,
we can handle the first two terms in the integral of (\ref{def1a7069}) as follows:
\begin{align}\label{e41ffd031}
& \int_{B_{2R}} \left( \frac{\Delta u (x\cdot\nabla u )\eta^2}{\sqrt{1+z}}
-\frac{(x\cdot\nabla u )(\nabla u \cdot\nabla z)\eta^2}{2(1+z)^{3/2}} \right) dx
\notag\\
&= \int_{B_{2R}} \operatorname{div}\left( \frac{\nabla u}{\sqrt{1+z}} \right) (x \cdot \nabla u )\eta^2 dx \notag\\
&= \int_{B_{2R}} 2 \eta\left( -\frac{\nabla u}{\sqrt{1+z}} (x \cdot \nabla u )
+ x\sqrt{1+z} \right) \cdot \nabla \eta dx \notag\\
& \qquad\qquad+\int_{B_{2R}} \left( - \frac{z}{\sqrt{1+z}}+n\sqrt{1+z} \right) \eta^2 dx
\end{align}
Furthermore, applying the divergence theorem, we have
\begin{align}\label{e4fbf921c}
\int_{B_{2R}} \frac{(x \cdot \nabla z)\eta^2}{2(1+z)^{3/2}}
&= - \int_{B_{2R}} x \cdot \nabla \left(\frac{1}{\sqrt{1+z}} \right) \eta^2 dx \notag\\
&=\int_{B_{2R}} \frac{\operatorname{div}\left( x \eta^2 \right)}{\sqrt{1+z}} dx \notag\\
&=\int_{B_{2R}} \left( \frac{n \eta^2}{\sqrt{1+z}}+\frac{2\eta(x \cdot \nabla \eta)}{\sqrt{1+z}} \right)dx
\end{align}

Substituting (\ref{e41ffd031}) and (\ref{e4fbf921c}) into (\ref{def1a7069}) and rearranging, we move the terms containing $\eta^2$ to the left-hand side of the inequality and the terms containing $\nabla\eta$ to the right-hand side.
Noting
This yields
\begin{align}\label{f5c25d11d}
\int_{B_{2R}} \frac{z \left((n-1) z+(n-2) \right)}{(1+z)^{3/2}} \eta^2
&\leq \int_{B_{2R}} \bigg( \frac{(x \cdot \nabla u )^2 |\nabla \eta|^2}{\sqrt{1+z}}
+ \frac{2(2+z)\eta}{(1+z)^{3/2}} (x \cdot \nabla u )(\nabla u \cdot \nabla \eta) \notag\\
&\qquad\qquad-\frac{2 z \eta}{\sqrt{1+z}} (x \cdot \nabla \eta)
- \frac{(x \cdot \nabla u )^2(\nabla u \cdot \nabla \eta)^2}{(1+z)^{3/2}} \bigg)dx.
\end{align}

Given $0 < \rho < R$, we choose the test function $\eta$ as follows:
\begin{align*}
\eta =
\left\{
\begin{array}{ll}
\rho^{-a} \quad & \text{if}\ |x|<\rho, \\
|x|^{-a} \phi \quad &\text{if}\ \rho\leq |x|<2R.
\end{array}
\right.
\end{align*}
Since \(\phi\equiv 1\) in \(B_\rho(0)\), the function \(\eta\) belongs to \(C_c^\infty(B_{2R})\). 
Noting $\nabla \eta = 0$ in $B_\rho(0)$, substituting $\eta$ into (\ref{f5c25d11d}) we get
\begin{align}\label{cf7ead00e}
&\frac{1}{\rho^{2a}} \int_{B_\rho} \frac{z \left((n-1) z+(n-2) \right)}{(1+z)^{3/2}} dx \notag\\
&\leq \int_{B_{2R} \backslash B_\rho} \bigg( \frac{(x \cdot \nabla u )^2 |\nabla \eta|^2}{\sqrt{1+z}}
+ \frac{2(2+z)\eta}{(1+z)^{3/2}} (x \cdot \nabla u )(\nabla u \cdot \nabla \eta) \notag\\
&\qquad\qquad\qquad\quad-\frac{2z\eta}{\sqrt{1+z}} (x \cdot \nabla \eta)
- \frac{(x \cdot \nabla u )^2(\nabla u \cdot \nabla \eta)^2}{(1+z)^{3/2}} \bigg)dx.
\end{align}

Noting that in $B_{2R}\backslash B_\rho$ there holds
\begin{align*}
\nabla \eta &= -a|x|^{-a-2}\phi x+|x|^{-a} \nabla \phi.
\end{align*}
It follows that
\begin{align*}
|\nabla \eta|^2 &= a^2 |x|^{-2a-2}\phi^2+|x|^{-2a} |\nabla \phi|^2 -2a |x|^{-2a-2} \phi (x\cdot\nabla\phi),
\end{align*}
\begin{align*}
\nabla u \cdot \nabla \eta &= -a|x|^{-a-2}\phi (x\cdot \nabla u)+|x|^{-a} (\nabla u \cdot \nabla \phi),
\end{align*}
and
\begin{align*}
x \cdot \nabla \eta &= -a|x|^{-a}\phi+|x|^{-a} (x \cdot \nabla \phi).
\end{align*}
Substituting these expressions into (\ref{cf7ead00e}),
merging and rearranging the terms yields
\begin{align}\label{d1292a215}
&\frac{1}{\rho^{2a}} \int_{B_\rho} \frac{z \left((n-1) z+(n-2) \right)}{(1+z)^{3/2}} dx
+ \int_{B_{2R}\backslash B_\rho} (L_1+L_2+L_3) \phi^2 dx \notag\\
&\leq
\int_{B_{2R}\backslash B_\rho}
\bigg(
-\frac{(x \cdot \nabla u )^2(\nabla u \cdot \nabla \phi)^2}{|x|^{2a}(1+z)^{3/2}}
+ \frac{(x \cdot \nabla u )^2 |\nabla \phi|^2}{|x|^{2a}\sqrt{1+z}} \notag\\
&\qquad\qquad\qquad+\frac{2a (x \cdot \nabla u )^3(\nabla u \cdot \nabla \phi)\phi}{|x|^{2(a+1)}(1+z)^{3/2}}
-\frac{2a (x \cdot \nabla u )^2(x \cdot \nabla \phi)\phi}{|x|^{2(1+a)}\sqrt{1+z}} \notag\\
&\qquad\qquad\qquad+\frac{2(2+z)(x \cdot \nabla u )(\nabla u \cdot \nabla \phi)\phi}{|x|^{2a}(1+z)^{3/2}}
- \frac{z(x\cdot\nabla \phi)}{|x|^{2a}\sqrt{1+z}}
\bigg)dx.
\end{align}
For the left-hand side of (\ref{d1292a215}), we have
\begin{align}\label{5cd8f2135}
\frac{1}{\rho^{2a}} \int_{B_\rho} \frac{z \left((n-1) z+(n-2) \right)}{(1+z)^{3/2}} dx
\geq \frac{C}{\rho^{2a}} \int_{B_\rho} \frac{z^2}{(1+z)^{3/2}} dx.
\end{align}
Then combining (\ref{d1292a215}) and (\ref{5cd8f2135}), we complete the proof of Lemma~\ref{lemma key grad est}.

\end{proof}

\section{Inner Regularity}

We give the proof of Theorem~\ref{theorem regularity}

\begin{proof}[Proof of Theorem~\ref{theorem regularity}]
Let $a$ be a dimensional constant to be specified later,
and $\phi\in C_c^\infty(B_2)$ be a cutoff function satisfying
\begin{align}\label{test function phi}
\left\{
\begin{array}{ll}
\phi(x) = 1 \quad & \text{if}\ |x|\leq 1,\\
0\leq \phi(x) \leq 1,\ |\nabla \phi(x) | \leq C \quad & \text{if}\ 1<|x|<2,\\
\phi(x) = 0 \quad & \text{if}\ |x|\geq 2.
\end{array}
\right.
\end{align}
We fix $R=1$ and $0<\rho<1$, and apply Lemma~\ref{lemma key grad est}.
Since $\nabla \phi$ is supported in $B_2 \backslash B_1$ and $(x \cdot \nabla u ) \leq |x| z^{1/2} $,
a straightforward estimate of the right-hand side of (\ref{key grad est}) gives
\begin{align*}
\text{RHS of}\ (\ref{key grad est})
&\leq C \int_{B_2\backslash B_1} \left( \frac{1}{|x|^{2a-2}}+\frac{1}{|x|^{2a-1}} \right) \left( \frac{z}{\sqrt{1+z}} + \frac{z^2}{(1+z)^{3/2}}  \right) dx \notag\\
&\leq C \int_{B_2\backslash B_1} \left( \frac{z}{\sqrt{1+z}} + \frac{z^2}{(1+z)^{3/2}}  \right) dx \notag\\
&\leq C \| \nabla u \|_{L^1(B_2)}
\end{align*}
Therefore Lemma~\ref{lemma key grad est} gives
\begin{align}\label{est with a}
\frac{1}{\rho^{2a}} \int_{B_\rho} \frac{z^2}{(1+z)^{3/2}} dx
+ \int_{B_2\backslash B_\rho} \left( L_1+L_2+L_3 \right) \phi^2 dx
\leq C \| \nabla u \|_{L^1(B_2)}
\end{align}
where $L_1$, $L_2$, and $L_3$ are defined in Lemma~\ref{lemma key grad est}.

We now distinguish two ranges of dimension.

{\bf Case $ 2\leq n \leq 5$:}
Choose $a=(n-1)/2$.
It follows that
\begin{align*}
L_1\geq 0,\quad L_2\geq 0,\quad L_3 = - \frac{z}{|x|^{2a} (1+z)^{3/2}}
\end{align*}
It follows that 
\begin{align*}
L_1+L_2+L_3 \geq -\frac{z}{|x|^{2a}(1+z)^{3/2}}.
\end{align*}

{\bf Case $n \geq 6$:}
Choose $a=2\sqrt{n-1}-2$.
From the explicit formulas for $L_2$ and $L_3$ we extract the leading terms in $|\nabla u|$:
\begin{align*}
L_2\geq a(2-a)\frac{z (x \cdot \nabla u )^2}{|x|^{2a+2}(1+z)^{3/2}} - C\frac{z}{|x|^{2a}(1+z)^{3/2}}
\end{align*}
\begin{align*}
L_3\geq (n-1-2a)\frac{z^2}{|x|^{2a}(1+z)^{3/2}} - C\frac{z}{|x|^{2a}(1+z)^{3/2}}
\end{align*}
Define the quantities
\begin{align}\label{def s t}
s := \frac{(x \cdot \nabla u )^2}{|x|^{a+2}(1+z)^{3/4}}, \quad
t := \frac{1}{|x|^a(1+z)^{3/4}}.
\end{align}
Then we can write
\begin{align*}
L_1+L_2+L_3 \geq \alpha s^2+\beta st+\gamma t^2 - C\frac{z}{|x|^{2a}(1+z)^{3/2}}
\end{align*}
where
\begin{align}\label{def alpha beta gamma}
\alpha = a^2,
\quad \beta = a(2-a)z,
\quad \gamma = (n-1-2a) z^2.
\end{align}
The discriminant of the quadratic form $s^2+\beta st+\gamma t^2 $ equals
\begin{align*}
\beta^2-4\alpha\gamma = a^2 z^2 (a (a+4)-4 n+8).
\end{align*}
With our choice $a=2\sqrt{n-1}-2$ one checks that $a(a+4)-4n+8=0$.
Hence the quadratic form is nonnegative:
\begin{align*}
\alpha s^2+\beta st+\gamma t^2 \geq 0.
\end{align*}
Consequently,
\begin{align*}
L_1+L_2+L_3 \geq - C\frac{z}{|x|^{2a}(1+z)^{3/2}}.
\end{align*}

In both cases, we obtain the same type lower bound for $L_1+L_2+L_3$, 
with the exponent $a$ depending on the dimension $n$.
Specifically, we define
\begin{align*}
a =
\left\{
\begin{array}{ll}
(n-1)/2 & \text{if}\ 2 \leq n \leq 5,\\
2\sqrt{n-1}-2 & \text{if}\ n \geq 6.
\end{array}
\right.
\end{align*}
With this choice of $a$, we get
\begin{align*}
L_1+L_2+L_3 \geq - C\frac{z}{|x|^{2a}(1+z)^{3/2}}.
\end{align*}
for some dimensional constant $C$. 
Since $z/(1+z)^{3/2}$ is bounded and $|x|^{-2a}$ is integrable on $B_2$ for $2a<n$, we have
\begin{align*}
\int_{B_2\backslash B_\rho} (L_1+L_2+L_3) \phi^2 dx \geq -C \int_{B_2\backslash B_\rho} \frac{1}{|x|^{2a}} dx \geq -C
\qquad \forall \ 0<\rho<1.
\end{align*}
Substituting this into (\ref{est with a}) we obtain 
\begin{align*}
\int_{B_\rho(0)} \frac{z^2}{(1+z)^{3/2}} dx
\leq C\rho^{2a} \left( 1+\| \nabla u \|_{L^1(B_2)} \right).
\end{align*}
Splitting the integral of $|\nabla u|$ according to whether $|\nabla u|\leq 1$ or $|\nabla u|>1$, we get
\begin{align}\label{6222894d2}
\int_{B_\rho(0)} |\nabla u| dx
&\leq \int_{B_\rho(0) \cap \{|\nabla u| \leq 1 \}} 1 dx
+ C \int_{B_\rho(0) \cap \{|\nabla u| >1 \}} \frac{z^2}{(1+z)^{3/2}} dx \notag\\
&\leq C(\rho^n+\rho^{2a})\left( 1+\| \nabla u \|_{L^1(B_2)} \right)
\end{align}
Since $2a<n$ for our choice of $a$ and $0<\rho<1$, the term $\rho^n$ is dominated by $\rho^{2a}$. Therefore
\begin{align*}
\int_{B_\rho(0)} |\nabla u| dx
\leq C\rho^{2a}\left( 1+\| \nabla u \|_{L^1(B_2)} \right) \qquad \forall \ 0<\rho<1.
\end{align*}

By translation invariance of the equation and the stability condition,
the same estimate holds for any ball $B_\rho(y)\subset B_{1/2}$ with
$0<\rho<1/2$. Hence, by the definition of Morrey spaces,
\[
\|\nabla u\|_{M^{p_n}(B_{1/2})}
\leq C \left( 1+\|\nabla u\|_{L^1(B_2)} \right),
\qquad p_n=\frac{n}{n-2a}.
\]


Finally, using scaling invariance together with a finite covering argument,
the above bound extends to $B_1$, yielding
\begin{align}\label{gradL1 ctrl GradMn}
\|\nabla u\|_{M^{p_n}(B_1)}
\leq C \left( 1+\|\nabla u\|_{L^1(B_2)} \right).
\end{align}
This completes the proof.
\end{proof}

Next we give the proof of Proposition~\ref{proposition 1}, which provides a sharper decay estimate in the radially symmetric case.

\begin{proof}[Proof of Proposition~\ref{proposition 1}]
Let $a = (n-1+\alpha)/2$, $R=1$, and $\phi$ be a cutoff function satisfying (\ref{test function phi}).
Since $u$ is radially symmetric, we have
\begin{align}\label{29f43fecf}
(x\cdot \nabla u)^2 = |x|^2 z
\end{align}
and a direct calculation gives
\begin{align}\label{radial L1+L2+L3}
L_1+L_2+L_3 &= \frac{(n-2)z^{2}+(-a^2+2a+n-2)z}{|x|^{2a}(1+z)^{3/2}},
\end{align}
where $L_1$, $L_2$, and $L_3$ are defined in Lemma~\ref{lemma key grad est}.
Since $n\geq 2$, the coefficient of $z^2$ is nonnegative, and hence the quadratic polynomial
$(n-2)z^{2}+(-a^2+2a+n-2)z$ is bounded from below for $z\geq 0$.
Also noting $|x|^{-2a}$ is integrable.
Hence there exists a constant $K=K(n,\alpha)$ such that
\begin{align*}
\int_{B_2\backslash B_\rho} (L_1+L_2+L_3) \phi^2 dx
\geq -K\int_{B_2\backslash B_\rho} \frac{1}{|x|^{2a}} dx
\geq -K
\end{align*}
for all $\rho\in(0,1)$.
Additionally, we set $\phi$ also be radial and nonincreasing in the radial variable.
In particular, we get
\begin{align}\label{c582d1d70}
x\cdot \nabla \phi = -|x| |\nabla \phi|,\qquad
\nabla u \cdot \nabla \phi = - (x \cdot \nabla u) \frac{|\nabla \phi|}{|x|}
\qquad \forall \ |x|\in(1,2).
\end{align}
Applying (\ref{c582d1d70}) and (\ref{29f43fecf}) again,
noting $2a<n$,
we obtain, after possibly enlarging the constant $K=K(n,\alpha)$
\begin{align*}
\text{RHS of}\ (\ref{key grad est})
&= \int_{B_2\backslash {B_1}}
\left(\frac{z |\nabla \phi|^2}{|x|^{2a-2}(1+z)^{3/2}}
+\frac{2(a-1) z \phi |\nabla \phi|}{|x|^{2a-1} (1+z)^{3/2}} \right)dx \\
&\leq K \int_{B_2\backslash {B_1}}
\left(\frac{|\nabla \phi|^2}{|x|^{2a-2}}
+\frac{\phi |\nabla \phi|}{|x|^{2a-1}} \right)dx \\
&\leq K.
\end{align*}
Hence the gradient estimate (\ref{key grad est}) becomes
\begin{align*}
\int_{B_\rho} \frac{z^2}{(1+z)^{3/2}} \phi^2 dx
\leq K \rho^{n-1+\alpha}\qquad \forall \ \rho\in(0,1).
\end{align*}
Noting $n-1+\alpha<n$, then using a similar splitting argument as in (\ref{6222894d2}) one may obtain
\begin{align*}
\int_{B_\rho} |\nabla u| dx
\leq K \rho^{n-1+\alpha}\qquad \forall \ \rho\in(0,1).
\end{align*}
which completes the proof.
\end{proof}

\begin{proof}[Proof of Corollary~\ref{corollary 1}]
Because $u$ is radial, we write $u(x) = u(r)$ with $r = |x|$, so that $|\nabla u(x)| = |u'(r)|$.
Using spherical coordinates, for any $0 < \rho < 1$ we have
\begin{align*}
\int_{B_\rho(0)} |\nabla u (x)| dx
= \omega_{n-1}\int_{0}^{\rho} |u'(s)| s^{n-1} ds
\end{align*}
where $\omega_{n-1}$ denotes the surface area of $S^{n-1}$.
Proposition~\ref{proposition 1} gives $\int_{B_\rho(0)} |\nabla u|dx \leq K\rho^{n-1+\alpha}$ with $K=K(n,\alpha)$, which yields
\begin{align*}
\int_{0}^{\rho} |u'(s)| s^{n-1} ds \leq \frac{K}{\omega_{n-1}} \rho^{n-1+\alpha},
\qquad 0<\rho<1.
\end{align*}

Now take $0<\rho_1<\rho_2<1$. Then
\begin{align*}
|u(\rho_2)-u(\rho_1)|
\leq \int_{\rho_1}^{\rho_2} |u'(s)|ds
\leq \rho_1^{1-n}\int_{\rho_1}^{\rho_2} |u'(s)| s^{n-1}ds
\leq \frac{K}{\omega_{n-1}}\rho_1^{1-n}\rho_2^{n-1+\alpha}
.
\end{align*}
Thus there exists constant $C(n,\alpha)$ such that 
\begin{align}\label{cauchy sequence est}
|u(\rho_2)-u(\rho_1)|
\leq C(n,\alpha)\Bigl(\frac{\rho_2}{\rho_1}\Bigr)^{\!n-1}\rho_2^{\alpha},
\qquad 0<\rho_1<\rho_2<1.
\end{align}

Take $r_k = 2^{-k}$ ($k \in \mathbb{N}$).
Putting $\rho_1=r_{k+1}$ and $\rho_2=r_k$ into (\ref{cauchy sequence est}), we obtain
\begin{align*}
|u(r_{k+1})-u(r_k)| \leq C(n,\alpha) 2^{n-1} (2^{-k})^\alpha = C(n,\alpha) 2^{-k\alpha}.
\end{align*}
Since $\alpha>0$, the series $\sum_{k \in \mathbb{N}} 2^{-k\alpha}$ converges,
Consequently, $\{u(r_k) \}_{k\rightarrow \infty}$ is a Cauchy sequence,
and the limit $M := \lim_{k\rightarrow \infty} u(r_k)$ exists and is finite.

It remains to show that $u(t)\to M$ as $t\rightarrow 0^+$.
Let $t \in (0,1/2)$ and choose $k$ such that $r_{k+1} < t \leq r_k$.
Setting $\rho_1=t$, $\rho_2=r_{k}$ in (\ref{cauchy sequence est}),
noting $\rho_2/\rho_1\leq 2$ and $r_k\leq 2t$, we obtain
\begin{align*}
|u(t)-u(r_{k})| \leq C(n,\alpha) 2^{n-1} (2t)^\alpha = C(n,\alpha) t^{\alpha}.
\end{align*}
Therefore
\begin{align*}
|u(t)-M|
\leq |u(t)-u(r_k)|+|u(r_k)-M|
\leq C(n,\alpha)t^{\alpha}+|u(r_k)-M|.
\end{align*}
As $t\to0^+$ we have $k\to\infty$, so $|u(r_k)-M|\to0$. 
Hence $|u(t)-M|\to0$.
Thus
\begin{align*}
\lim_{t\rightarrow 0^+} u(t) = M,
\end{align*}
which means that the origin is at most a removable singularity of $u$.
\end{proof}

The following proof shows that the critical exponent $p_n$ obtained in Theorem~\ref{theorem regularity}
cannot be improved: 
for every $\alpha>0$ one can produce a stable solution whose gradient fails to belong to the slightly larger space $M^{n+\alpha}$.
The example is of the form $u(x)=|x_1|^{1-a}$ with $a$ close to $1$.
Such a function has a gradient that blows up on the hyperplane $\{x_1=0\}$ exactly at the rate $|\nabla u(x)| \sim |x_1|^{-a}$, 
which is just outside the Morrey space $M^{n+\alpha}$ when $a$ is chosen appropriately.

\begin{proof}[Proof of Theorem~\ref{theorem bad regularity solution}]
Let $a\in (1/2,1)$ and define $u(x) = |x_1|^{1-a}$.
Then we have
\begin{align*}
\nabla u = \left( \frac{(1-a) \mathrm{sgn}(x_1)} {|x_1|^a}, 0, 0, \cdots \right),
\qquad |\nabla u| = \frac{1-a}{|x_1|^a}.
\end{align*}
A direct calculation shows that $\nabla u\in M^{n/a}(B_1)$,
but $\nabla u \notin M^{n/a+b}(B_1)$ for any $b>0$.
Choosing
\begin{align*}
a=\frac{n}{n+\alpha/2},\qquad b=\frac{\alpha}{2}
\end{align*}
we then obtain $\nabla u \notin M^{n+\alpha}(B_1)$.
Hence, it suffices to show that such $u$ is a stable solution of (\ref{equation}) for a suitable nonlinearity $f\in C^1(\mathbb{R})$.

{\bf Step 1 (The associated nonlinearity).}
A direct computation shows
\begin{align*}
\operatorname{div}\left( \frac{\nabla u}{\sqrt{1+|\nabla u|^2}} \right)
= -\frac{a(1-a)|x_1|^{-a-1}}{\left(1+(1-a)^2 |x_1|^{-2 a}\right)^{3/2}}
\end{align*}
Since $u(x)=|x_1|^{1-a}$, we have $|x_1| = u^{1/(1-a)}$. Therefore $u$ satisfies
\begin{align*}
\operatorname{div}\left( \frac{\nabla u}{\sqrt{1+|\nabla u|^2}} \right) = -f(u),
\end{align*}
where
$f=f_a\in C^1(\mathbb{R})$ defined by
\begin{align*}
f(t) =
\left\{
\begin{array}{ll}
\displaystyle
\frac{a(1-a)t^{\frac{a+1}{a-1}}}{\left(1+(1-a)^2 t^{\frac{2 a}{a-1}}\right)^{3/2}}
\qquad &\text{if}\ t>0, \\
0\qquad &\text{if}\ t\leq 0.
\end{array}
\right.
\end{align*}
Moreover, one may check that
\begin{align*}
u \in W^{1,1}(B_{1}),
\qquad \frac{\nabla u}{\sqrt{1+|\nabla u|^2}} \in L_{loc}^\infty (B_{1}),
\qquad f(u) \in L_{loc}^1 (B_{1}).
\end{align*}
Consequently $u$ is a weak solution of equation (\ref{equation}) with this $f$ in $B_1$.

{\bf Step 2 (Reduction of the stability condition).}
Differentiating $f$ yields
\begin{align*}
f'(u) = \frac{a \left( (1-a)^2 (2a-1)
-(1+a)u^{\frac{2 a}{1-a}}
\right)u^{\frac{2 (a+1)}{a-1}}}
{\left(1+(1-a)^2 u^{\frac{2 a}{a-1}}\right)^{5/2}}
\end{align*}
It follows that
\begin{align*}
f'(u) &=
\frac{a \left( (1-a)^2 (2a-1)-(1+a)|x_1|^{2 a} \right)|x_1|^{-2 (a+1)}}
{\left(1+(1-a)^2 |x_1|^{-2 a}\right)^{5/2}} \\
&\leq
\frac{a (1-a)^2 (2a-1)|x_1|^{-2 (a+1)}}
{\left(1+(1-a)^2 |x_1|^{-2 a}\right)^{5/2}} \\
&\leq
\frac{a (2 a-1)}{(1-a)^3}|x_1|^{3 a-2}
\end{align*}
Integrating above, we have
\begin{align}\label{example lhs of stable}
\int_{|x|<t} f'(u) \varphi^2 dx \leq \frac{a (2 a-1)}{(1-a)^3} \int_{|x|<t}|x_1|^{3 a-2} \varphi^2 dx
\qquad \forall \ 0<t\leq 1.
\end{align}
On the other hand, using
$(\nabla u \cdot \nabla \varphi) \leq |\nabla u| |\nabla \varphi|$,
we have
\begin{align*}
\int_{|x|<t} \left( \frac{|\nabla \varphi|^2}{\sqrt{1+z}} - \frac{(\nabla u \cdot \nabla \varphi)^2}{(1+z)^{3/2}} \right)
&\geq \int_{|x|<t} \frac{|\nabla \varphi|^2}{(1+z)^{3/2}} \notag\\
&= \int_{|x|<t}
\frac{|\nabla \varphi|^2}{\left(1+ (1-a)^2 |x_1|^{-2a}\right)^{3/2}}
\qquad \forall \ 0<t\leq 1.
\end{align*}
Fix any $0<\lambda < 1$ and set
\begin{align*}
r_1 = \Bigl[(1-a)^2(\lambda^{-2/3}-1)\Bigr]^{1/(2a)}.
\end{align*}
Whenever \(|x_1| \leq r_1\), we have \(|x_1|^{2a} \leq (1-a)^2(\lambda^{-2/3}-1)\), and therefore
\begin{align*}
|x_1|^{2a}+(1-a)^2 \leq (1-a)^2\lambda^{-2/3}.
\end{align*}
Then we obtain
\begin{align*}
\frac{1}{\bigl(1+(1-a)^2|x_1|^{-2a}\bigr)^{3/2}} = \frac{|x_1|^{3a}}{\bigl(|x_1|^{2a}+(1-a)^2\bigr)^{3/2}}
\geq \lambda\frac{|x_1|^{3a}}{(1-a)^3}.
\end{align*}
It follows that
\begin{align}\label{example rhs of stable}
\int_{|x|<t} \left( \frac{|\nabla \varphi|^2}{\sqrt{1+z}} - \frac{(\nabla u \cdot \nabla \varphi)^2}{(1+z)^{3/2}} \right)
\geq \frac{\lambda}{(1-a)^3} \int_{|x|<t} |x_1|^{3a} |\nabla \varphi|^2 dx
\quad \forall \ 0<t\leq r_1.
\end{align}

Finally, let \(r_0 = r_0(\lambda) := \min\{1,r_1\}\). Then both estimates hold simultaneously for all \(|x|\leq r_0\).
Combining (\ref{example lhs of stable}) and (\ref{example rhs of stable}),
it can be seen that if one wants to prove
stability condition
\begin{align}\label{stable condition to proof}
\int_{B_{r_0}} f'(u)\varphi^2 dx
\leq \int_{B_{r_0}} \left( \frac{|\nabla \varphi|^2}{\sqrt{1+z}} - \frac{(\nabla u \cdot \nabla \varphi)^2}{(1+z)^{3/2}} \right) dx.
\end{align}
it is sufficient to prove
\begin{align}\label{to proof 2}
\int_{B_{r_0}}|x_1|^{3 a-2} \varphi^2 dx
\leq \frac{\lambda}{a (2 a-1)} \int_{B_{r_0}} |x_1|^{3a} |\nabla \varphi|^2 dx,
\end{align}
for all $\varphi\in C_c^\infty(B_{r_0}(0))$.

{\bf Step 3 (Verification of the stability condition).}
We set
\begin{align*}
\lambda=\frac{4a(2a-1)}{(3a-1)^2}.
\end{align*}
Because $1/2 < a < 1$, we have $0<\lambda<1$,
so this choice is admissible in Step~2.
Write $x = (x_1, x')\in \mathbb{R}\times\mathbb{R}^{n-1}$.
Applying the one-dimensional weighted Hardy inequality (see Lemma~\ref{Lemma Appendix})
\begin{align*}
\int_{-r_{0}}^{r_{0}}|t|^{3a-2} \psi(t)^{2}dt
\leq \Bigl(\frac{2}{3a-1}\Bigr)^{2}
\int_{-r_{0}}^{r_{0}} |t|^{3a}|\psi'(t)|^{2}dt
\qquad(\psi\in C_{c}^{\infty}(-r_{0},r_{0}))
\end{align*}
with $\psi(t)=\varphi(t,x')$ and integrating with respect to $x'$, we obtain
\begin{align*}
\int_{B_{r_0}} |x_{1}|^{3a-2}\varphi^{2}dx
&\leq \Bigl(\frac{2}{3a-1}\Bigr)^{2}
\int_{B_{r_0}} |x_{1}|^{3a}|\partial_{x_{1}}\varphi|^{2}dx \\[2pt]
&\leq \Bigl(\frac{2}{3a-1}\Bigr)^{2}
\int_{B_{r_0}} |x_{1}|^{3a}|\nabla\varphi|^{2}dx.
\end{align*}
Thanks the choice of $\lambda$, 
the last estimate is exactly (\ref{to proof 2}).
Hence $u$ satisfies the stability inequality (\ref{stable condition to proof}) for all $\varphi\in C_{c}^{\infty}(B_{r_{0}})$.
Thus $u$ is a stable solution of (\ref{equation}) in $B_{r_{0}}$.

Finally, set $u_{c}(x)=u(cx)/c$ with $c=r_{0}$. A straightforward change of variables shows that $u_{c}$ satisfies (\ref{equation}) in $B_{1}$ with a rescaled nonlinearity $f_{c}(t):=cf(ct)\in C^{1}(\mathbb{R})$, and the stability condition is preserved under this scaling. Moreover, $\nabla u\notin M^{n+\alpha}(B_{r_0})$ implies $\nabla u_{c}\notin M^{n+\alpha}(B_{1})$. Consequently $u_{c}$ is a stable solution in $B_{1}$ whose gradient does not belong to $M^{n+\alpha}(B_{1})$, which completes the proof.
\end{proof}

\section{Liouville Theorem}

\begin{proof}[Proof of Theorem~\ref{theorem liouville}]
Let $\phi$ be a smooth cutoff function satisfying
\begin{align}\label{def phi}
\left\{
\begin{array}{ll}
\phi(x) = 1, \quad & \text{if}\ |x|\leq R,\\
0\leq \phi(x) \leq 1\ \text{and}\ |\nabla \phi(x) | \leq C/R \quad & \text{if}\ R<|x|<2R,\\
\phi(x) = 0, \quad & \text{if}\ |x|\geq 2R.
\end{array}
\right.
\end{align}
Since $\nabla u$ is supported in $B_{2R}\backslash B_{R}$, we have
\begin{align*}
\text{RHS of}\ (\ref{key grad est})
\leq C \int_{B_{2R}\backslash {B_R}} \frac{z}{|x|^{2a}\sqrt{1+z}} dx
\leq \frac{C}{R^{2a}} \int_{B_{2R}\backslash {B_R}} z dx
\end{align*}
Then Lemma~\ref{lemma key grad est} gives
\begin{align}\label{72bed3d97}
\frac{1}{\rho^{2a}} \int_{B_\rho} \frac{z^2}{(1+z)^{3/2}} dx
+ \int_{B_{2R}\backslash B_\rho} \left( L_1+L_2+L_3 \right) \phi^2 dx
\leq \frac{C}{R^{2a}} \int_{B_{2R}\backslash {B_R}} z dx
\end{align}
Next, we consider two cases.

{\bf Case $2\leq n \leq 10$:}
In this case we set
\begin{align*}
a = \frac{n-2}{2}.
\end{align*}
Introducing $s,t$ and $\alpha,\beta,\gamma$ as in
(\ref{def s t}) and (\ref{def alpha beta gamma}).
Noting $a\leq 4$, this will gives
\begin{align*}
L_1 = \alpha s^2, \quad L_2 \geq \beta st, \quad L_3 = \gamma t^2.
\end{align*}
We compute the discriminant of the quadratic form $\alpha s^2+\beta st+\gamma t^2$:
\begin{align*}
\beta^2-4\alpha\gamma = a^2 z^2 (a (a+4)-4 n+8) =\frac{1}{16} (n-2)^2 \left(n^2-12 n+20\right) z^2
\end{align*}
Since $2\leq n \leq 10$, it follows that $\beta^2-4\alpha\gamma\leq 0$,
which yields
\begin{align*}
L_1+L_2+L_3 \geq \alpha s^2+\beta st+\gamma t^2 \geq 0.
\end{align*}

{\bf Case $n\geq 11$:}
In this case, we set
\begin{align*}
a=\sqrt{n-1}+1.
\end{align*}
Adopt the same notion $\alpha$, $\beta$, $\gamma$, $s$ and $t$ in above case,
noting $a\geq 4$ and $a\leq (n-2)/2$
we obtain
\begin{align*}
L_1= \alpha s^2,
\quad L_2 \geq \beta st+\frac{a(4-a) z}{|x|^{2a}(1+z)^{3/2}},
\quad L_3 = \gamma t^2+\frac{(n-2-2a) z}{|x|^{2a}(1+z)^{3/2}}
\end{align*}
Those together gives
\begin{align*}
L_1+L_2+L_3 &\geq \alpha s^2+\beta st+\gamma t^2+\frac{(-a^2+n-2+2a) z}{|x|^{2a}(1+z)^{3/2}} \\
&=s^2+\beta st+\gamma t^2
\end{align*}
where we have used that $a=\sqrt{n-1}+1$ implies $-a^2+n-2+2a=0$.
directly calculate shows
\begin{align*}
\beta^2-4\alpha\gamma = a^2 z^2 (a (a+4)-4 n+8)
=3 \left(-n+2 \sqrt{n-1}+4\right) \left(n+2 \sqrt{n-1}\right) z^2
\end{align*}
Noting $n\geq 11$ implies $-n+2 \sqrt{n-1}+4\leq 0$, which gives $\beta^2-4\alpha\gamma\leq 0$.
Then we obtain
\begin{align*}
L_1+L_2+L_3\geq \alpha s^2+\beta st+\gamma t^2 \geq 0.
\end{align*}
Define
\begin{align*}
a = \left\{
\begin{array}{ll}
(n-2)/2\quad & \text{if}\ 2\leq n\leq 10,\\
\sqrt{n-1}+1\quad & \text{if}\ n\geq 11.
\end{array}
\right.
\end{align*}
Then together both cases, we have $L_1+L_2+L_3\geq 0$.
Then from (\ref{72bed3d97}), we have
\begin{align}\label{22f651df6}
\frac{1}{\rho^{2a}} \int_{B_\rho} \frac{z^2}{(1+z)^{3/2}} dx
\leq \frac{C}{R^{2a}} \int_{B_{2R}\backslash {B_R}} z dx
\end{align}

{\bf Proof of Part (1).}
Assume that $u$ is nonconstant.
Then there exists some $\rho>0$ for which the left-hand side
of (\ref{22f651df6}) is positive
(otherwise $|\nabla u|\equiv0$ in $\mathbb{R}^n$ implies $u$ is a constant).
Fix such a $\rho$. From (\ref{22f651df6}) we obtain
\begin{align*}
\frac{1}{R^{2a}} \int_{B_{2R}\setminus B_{R}} z dx \geq c
\end{align*}
for some $c>0$ and all sufficiently large $R$. Because $|B_{2R}\setminus B_{R}| \sim R^{n}$, this implies
\begin{align*}
\fint_{B_{2R}\setminus B_{R}} z dx \geq cR^{2a-n}
\end{align*}
Recalling that $2a-n = 2q_n$, we arrive at (\ref{liouville theorem nonconstant bound}).

{\bf Proof of Part (2).}
This statement follows immediately,
since the growth condition (\ref{liouville theorem rigidity condition}) contradicts
(\ref{liouville theorem nonconstant bound}).
\end{proof}

\begin{proof}[Proof of Theorem~\ref{liouville theorem radial}]
Let
\begin{align*}
a=\sqrt{n-1}+1.
\end{align*}
Because $u$ is radial, identity (\ref{radial L1+L2+L3}) holds.
With our choice of $a$ we have $-a^2+2a+n-2=0$ and therefore
\begin{align*}
L_1+L_2+L_3\geq 0.
\end{align*}
Take a radial cutoff function $\phi$ satisfying (\ref{def phi}) and non-increasing in the radial variable.
For such a $\phi$ the relations (\ref{c582d1d70}) and (\ref{29f43fecf}) are valid.
Substituting them into the right-hand side of (\ref{key grad est}) gives
\begin{align*}
\text{RHS of}\ (\ref{key grad est})
= \int_{B_{2R}\setminus B_{R}}
\bigg(\frac{z|\nabla\phi|^{2}}{|x|^{2a-2}(1+z)^{3/2}}
+\frac{2(a-1)z\phi|\nabla\phi|}{|x|^{2a-1}(1+z)^{3/2}}\bigg)dx.
\end{align*}
Since $|\nabla\phi|\leq C/R$ and the integrand is supported in $B_{2R}\setminus B_{R}$,
we obtain
\begin{align*}
\text{RHS of}\ (\ref{key grad est})
\leq \frac{C}{R^{2a}} \int_{B_{2R}\setminus B_{R}}
\frac{z}{(1+z)^{3/2}}dx.
\end{align*}
Now apply Lemma~\ref{lemma key grad est} with this $\phi$. Using the nonnegativity of
$L_{1}+L_{2}+L_{3}$ we deduce
\begin{align}\label{edc14a382}
\frac{1}{\rho^{2a}}\int_{B_{\rho}}
\frac{z^2}{(1+z)^{3/2}}dx
\leq \frac{C}{R^{2a}} \int_{B_{2R}\setminus B_{R}}
\frac{z}{(1+z)^{3/2}}dx
\end{align}
for every $0<\rho<R$.

{\bf Proof of part (1).}
Assume that $u$ is nonconstant.
Then there exists $\rho>0$ such that left-hand side of (\ref{edc14a382}) is positive.
Fix such a $\rho$. From (\ref{edc14a382}) we deduce
\begin{align*}
\frac{1}{R^{2a}}
\int_{B_{2R}\setminus B_{R}}
\frac{z}{(1+z)^{3/2}}dx
\ge c
\end{align*}
for some $c>0$ and all sufficiently large $R$.
Since $|B_{2R}\setminus B_R| \sim R^n$, we obtain
\begin{align}\label{44883fc13}
\fint_{B_{2R}\setminus B_{R}}
\frac{z}{(1+z)^{3/2}}dx
\ge c R^{2a-n}.
\end{align}
Recalling that $2a-n = -n+2\sqrt{n-1}+2$,
we arrive at (\ref{liouville theorem radial nonconstant bound}).

{\bf Proof of part (2).}
We now show that each of the three conditions forces $u$ to be constant.

\begin{itemize}
\item
{Case $2\leq n\leq 6$:}
Noting $2a-n=-n+2\sqrt{n-1}+2>0$, the right-hand side of (\ref{44883fc13}) is large as $R\rightarrow \infty$.
However $z/(1+z)^{3/2}$ in left-hand side is bounded. 
So growth (\ref{44883fc13}) can never be satisfied in this case, which implies $u$ is a constant.
\item
{Case $n \geq 7$
and
\begin{align*}
|\nabla u(x)| = o(|x|^{-n/2+\sqrt{n-1}+1})\qquad \text{as}\ |x|\rightarrow+\infty.
\end{align*}
}
Using elementary inequality
\begin{align*}
\frac{z}{(1+z)^{3/2}} \leq z,
\end{align*}
from (\ref{edc14a382}) we obtain
\begin{align*}
\frac{1}{\rho^{2a}} \int_{B_\rho} \frac{z^2}{(1+z)^{3/2}} dx
\leq \frac{C}{R^{2a-n}} \fint_{B_{2R}\backslash {B_R}} |\nabla u|^2 dx
\end{align*}
Thanks to the assumed growth condition,
it follows that the right-hand side tends to $0$ as $R \to \infty$, which yields
\begin{align*}
\int_{B_\rho} \frac{z^2}{(1+z)^{3/2}} dx = 0.
\end{align*}
Since $\rho>0$ is arbitrary, we conclude $\nabla u = 0$ in $\mathbb{R}^n$, i.e., $u$ is constant.
\item
{Case $n \geq 7$
and
\begin{align*}
|\nabla u(x)|^{-1} = o(|x|^{-n+2\sqrt{n-1}+2})
\qquad \text{as}\ |x|\rightarrow+\infty.
\end{align*}
}
We use elementary inequality
\begin{align*}
\frac{z}{(1+z)^{3/2}} \leq \frac{1}{z^{1/2}}.
\end{align*}
From (\ref{edc14a382}) we obtain
\begin{align*}
\frac{1}{\rho^{2a}} \int_{B_\rho} \frac{z^2}{(1+z)^{3/2}} dx
\leq \frac{C}{R^{2a-n}} \fint_{B_{2R}\backslash {B_R}} \frac{1}{|\nabla u|} dx.
\end{align*}
Following the same procedure as in the previous case, we deduce that $u$ is constant.
\end{itemize}
\end{proof}

\section{Appendix}

Next, we intend to prove a one-dimensional weighted Hardy inequality (\ref{d397dd15f}).
It differs from the classical Hardy inequality
\begin{align}\label{6a6ac9888}
\int_0^{1} \frac{\varphi(t)^2}{t^2} dt \leq 4 \int_0^{1} \varphi'(t)^2 dt, \qquad \forall \ \varphi \in C_c^\infty(0,1).
\end{align}
in that the test functions $\varphi$ in (\ref{6a6ac9888}) cannot have support containing the origin, since $1/t^2$ is not integrable near $0$. In contrast, the test functions in (\ref{d397dd15f}) can have support containing the origin, because $|t|^{\beta-2}$ is not singular at $0$.

This means we cannot directly derive Lemma~\ref{Lemma Appendix} by simple transformations from the classical Hardy inequality. However, the proof of this lemma is similar in spirit to that of the Hardy inequality. This result might have appeared somewhere in the literature, but we could not find a reference. For the sake of completeness, we provide the proof here.

\begin{lemma}\label{Lemma Appendix}
Let \(\varphi \in C_c^\infty(-1,1)\) and \(\beta>1\). Then
\begin{align}\label{d397dd15f}
\int_{-1}^{1} |t|^{\beta-2} \varphi^2(t) dt \leq \left( \frac{2}{\beta-1} \right)^2 \int_{-1}^{1} |t|^{\beta} |\varphi'(t)|^2 dt.
\end{align}
\end{lemma}

\begin{proof}[Proof of Lemma~\ref{Lemma Appendix}]
We start with the integral on the left-hand side:
\begin{align*}
\int_{-1}^{1} \varphi^2(t) |t|^{\beta-2} dt &= \frac{-1}{\beta-1} \int_{-1}^{0} \varphi^2(t) d(-t)^{\beta-1}+\frac{1}{\beta-1} \int_{0}^{1} \varphi^2(t) d t^{\beta-1}.
\end{align*}
Integrating by parts, the boundary terms vanish due to the compact support of \(\varphi\), yielding
\begin{align*}
&= \frac{-1}{\beta-1} \varphi^2(t) (-t)^{\beta-1} \Big|_{-1}^{0}+\frac{1}{\beta-1} \int_{-1}^{0} (-t)^{\beta-1} d(\varphi^2(t)) \\
&\qquad + \frac{1}{\beta-1} \varphi^2(t) t^{\beta-1} \Big|_{0}^{1} - \frac{1}{\beta-1} \int_{0}^{1} t^{\beta-1} d(\varphi^2(t)) \\
&= \frac{2}{\beta-1} \int_{-1}^{0} (-t)^{\beta-1} \varphi'(t) \varphi(t) dt - \frac{2}{\beta-1} \int_{0}^{1} t^{\beta-1} \varphi'(t) \varphi(t) dt.
\end{align*}
Taking the absolute value and bounding,
\begin{align*}
&\leq \frac{2}{\beta-1} \int_{-1}^{1} |t|^{\beta-1} |\varphi'(t)| |\varphi(t)| dt \\
&\leq \frac{2}{\beta-1} \left( \int_{-1}^{1} |t|^{\beta-2} \varphi^2(t) dt \right)^{1/2} \left( \int_{-1}^{1} |t|^{\beta} |\varphi'(t)|^2 dt \right)^{1/2},
\end{align*}
where the last step follows from the Cauchy--Schwarz inequality. Denoting the left-hand side by \(I\), we have \(I \leq \frac{2}{\beta-1} I^{1/2} J^{1/2}\), where \(J = \int_{-1}^{1} |t|^{\beta} |\varphi'(t)|^2 dt\). Assuming \(I > 0\), dividing both sides by \(I^{1/2}\) gives \(I^{1/2} \leq \frac{2}{\beta-1} J^{1/2}\), and squaring both sides yields the desired inequality. If \(I = 0\), the inequality holds trivially.
\end{proof}

\textbf{Acknowledgements.}
The author would like to thank the anonymous referees for their careful reading
of the manuscript and for many valuable comments and suggestions, which have
led to a substantial improvement of the paper.

\textbf{Conflict of Interest Statement.}
The author declares that there is no conflict of interest.

\textbf{Data Availability Statement.}
No datasets were generated or analysed during the current study.

\end{document}